\documentclass[12pt,a4paper,reqno]{amsart}
\usepackage[english]{babel}
\usepackage[applemac]{inputenc}
\usepackage[T1]{fontenc}
\usepackage{palatino}
\usepackage{amsmath}
\usepackage{amssymb}
\usepackage{amsthm}
\usepackage{amsfonts}
\usepackage{graphicx}
\usepackage{mathtools}

\usepackage[colorlinks = true, citecolor = black]{hyperref}
\author{Tuomas Orponen}
\title{Improving Kaufman's exceptional set estimate for packing dimension}
\address{University of Helsinki, Department of Mathematics and Statistics}
\thanks{T.O. is supported by the Academy of Finland through the grant Restricted families of projections and connections to
Kakeya type problems.}
\email{tuomas.orponen@helsinki.fi}

\newcommand{\R}{\mathbb{R}}

\newcommand{\N}{\mathbb{N}}

\newcommand{\Z}{\mathbb{Z}}

\newcommand{\calT}{\mathcal{T}}
\newcommand{\calL}{\mathcal{L}}

\newcommand{\calH}{\mathcal{H}}

\newcommand{\calB}{\mathcal{B}}

\newcommand{\spt}{\operatorname{spt}}
\newcommand{\Hd}{\dim_{\mathrm{H}}}
\newcommand{\Pd}{\dim_{\mathrm{p}}}
\newcommand{\Bd}{\overline{\dim}_{\mathrm{B}}}

\newcommand{\spa}{\operatorname{span}}
\newcommand{\diam}{\operatorname{diam}}

\newcommand{\dist}{\operatorname{dist}}

\numberwithin{equation}{section}

\theoremstyle{plain}
\newtheorem{thm}[equation]{Theorem}
\newtheorem{conjecture}[equation]{Conjecture}

\newtheorem{cor}[equation]{Corollary}
\newtheorem{proposition}[equation]{Proposition}
\newtheorem{question}{Question}

\theoremstyle{definition}

\newtheorem{definition}[equation]{Definition}

\theoremstyle{remark}
\newtheorem{remark}[equation]{Remark}

\begin{document}

\begin{abstract} Given $0 < s < 1$, I prove that there exists a constant $\epsilon = \epsilon(s) > 0$ such that the following holds. Let $K \subset \R^{2}$ be a Borel set with $\calH^{1}(K) > 0$, and let $E_{s}(K) \subset S^{1}$ be the collection of unit vectors $e$ such that 
\begin{displaymath} \Pd \pi_{e}(K) \leq s. \end{displaymath}
Then $\Hd E_{s}(K) \leq s - \epsilon$.
\end{abstract}

\maketitle

\section{Introduction}

This paper is concerned with a classical question in fractal geometry: how do orthogonal projections affect the dimension of planar sets? For a reader not familiar with the area, I recommend the recent survey \cite{FFJ} of Falconer, Fraser and Jin. In this introduction, I only describe some results most relevant to the new material. 

Fix $0 \leq s \leq 1$, and let $K \subset \R^{2}$ be a Borel set of Hausdorff dimension $\Hd K \geq s$. In 1968, Kaufman \cite{Ka} proved, improving an earlier result of Marstrand \cite{Ma} from 1954, that
\begin{equation}\label{kaufman} \Hd \{e \in S^{1} : \Hd \pi_{e}(K) < s\} \leq s. \end{equation}
Here $\pi_{e} \colon \R^{2} \to \R$ is the orthogonal projection $\pi_{e}(x) = x \cdot e$. Under the assumption $\Hd K \geq s$, Kaufman's bound \eqref{kaufman} is sharp: in 1975, Kaufman and Mattila \cite{KM} constructed an explicit compact set $K \subset \R^{2}$ with $\Hd K = s$ such that
\begin{equation}\label{sharpness} \Hd \{e : \Hd \pi_{e}(K) < s\} = s. \end{equation}
Under the assumption $\Hd K \geq t > s$, the sharpness of \eqref{kaufman} is an open problem. The following improvement is conjectured (in (1.8) of \cite{Ob}, for instance):
\begin{conjecture}\label{mainC} Assume that $0 \leq t/2 \leq s \leq t \leq 1$ and $\Hd K \geq t$. Then
\begin{equation}\label{mainC2} \Hd \{e \in S^{1} : \Hd \pi_{e}(K) < s\} \leq 2s - t. \end{equation}
\end{conjecture} 
It is known that the right hand side of \eqref{mainC2} cannot be further improved. 

So, what are the partial results for Conjecture \ref{mainC}, and why is the problem worth studying? The case $s \approx t/2$ attracted a great deal of attention around 2003, when Edgar and Miller \cite{EM} and Bourgain \cite{Bo1} independently proved the Erd\H{o}s-Volkmann ring conjecture. The conjecture -- now a theorem -- states that $\R$ does not contain Borel subrings of Hausdorff dimension $r \in (0,1)$. To explain the connection to Conjecture \ref{mainC}, assume for a moment that there existed a Borel ring $B \subset \R$ with $\Hd B = r$ for some $0 < r \leq 1/2$. Then $\Hd (B \times B) \geq 2r$, and $B + hB \subset B$ for all $h \in B$. This implies that
\begin{equation}\label{form33} \Hd \{e \in S^{1} : \Hd \pi_{e}(B \times B) \leq r\}) \geq r, \end{equation}
severely violating \eqref{mainC2} for $s$ close to $r$ and $t = 2r$. Thus, Conjecture \ref{mainC} is stronger than the ring conjecture. In fact, the ring conjecture is no stronger than "a slight improvement over Kaufman's bound \eqref{kaufman} in the case $s = t/2$". To see this, note that Kaufman's bound \eqref{kaufman} implies
\begin{displaymath} \Hd \{e \in S^{1} : \Hd \pi_{e}(B \times B) < r\} \leq r, \end{displaymath}
which barely fails to contradict \eqref{form33}. So, a result of the form
\begin{equation}\label{form34} \Hd \{e : \Hd \pi_{e}(K) \leq r\} \leq r - \epsilon \end{equation}
for $0 < r \leq (\Hd K)/2$ would already be strong enough to settle the ring conjecture, and this is significantly weaker than Conjecture \ref{mainC}. Bourgain's approach to the ring problem gives \eqref{form34}, and in fact something quite a bit better, which sits between \eqref{form34} and the conjectured bound \eqref{mainC2}: Theorem 3 in \cite{Bo} implies that
\begin{equation}\label{form35} \Hd \{e : \Hd \pi_{e}(K) \leq s\} \searrow 0 \end{equation}
as $s \searrow (\Hd K)/2$. The biggest caveat is that \eqref{form35} says nothing about values of $s$ far from $(\Hd K)/2$. For instance, assuming that $\Hd K = 1$, the best bound for $\Hd \{e : \Hd \pi_{e}(K) \leq 3/4\}$ remains the one given by Kaufman's bound \eqref{kaufman}, namely $\Hd \{e : \Hd \pi_{e}(K) \leq 3/4\} \leq 3/4$. 

The ring conjecture was not the only motivation for Bourgain's work \cite{Bo} in 2003. Two years earlier, Katz and Tao \cite{KT} had proved that the case $t = 1/2$ of the ring conjecture is "roughly" equivalent (more precisely: equivalent at the level of certain "discretised" versions) to obtaining small improvements in Falconer's distance set problem and the $(1/2)$-Furstenberg set problem. I recall the latter question
\begin{question}[Furstenberg set problem]\label{furstenberg} Assume that a set $K \subset \R^{2}$ has the property that for every $e \in S^{1}$, there exists a line of the form $L_{a,e} := a + \spa(e)$, $a \in \R^{2}$, such that $\Hd (K \cap L_{a,e}) \geq s$. Such a set $K$ is called an $s$-Furstenberg set. How small can the dimension of an $s$-Furstenberg set be? \end{question}
Until Bourgain's work in 2003, the best result on Question \ref{furstenberg} was due to Wolff \cite{Wo}, who proved that $\Hd K \geq \max\{s + 1/2,2s\}$ for every $s$-Furstenberg set $K$. In the case $s = 1/2$, Bourgain could improve Wolff's result by a small absolute constant $c > 0$, namely showing that $\Hd K \geq 1 + c$. To sum up, a slight improvement of the type \eqref{form34} for Kaufman's bound \eqref{kaufman} in the case $\Hd K = 1$ and $s = 1/2$ is stronger than the case $t = 1/2$ of the ring conjecture, which is, further, "roughly" equivalent to proving an improvement for the dimension of $(1/2)$-Furstenberg sets. This is not a rigorous argument, but it is a fair guideline.

How about $s$-Furstenberg sets for $s \in (1/2,1)$? For $s$ "very close" to $1/2$, Bourgain's approach still gives an improvement over the Wolff bound, just as \eqref{form35} gives an improvement to Kaufman's bound for $s$ "very close" to $1/2$. But for $s = 3/4$, say, the best dimension bound for $s$-Furstenberg sets remains Wolff's estimate $\min\{1/2 + s,2s\} = 2s$. And, heuristically, improving Kaufman's bound for $\Hd K = 1$, and any $s \in (1/2,1)$, is "close" to improving Wolff's $2s$-bound for the \textbf{same} value of $s$. Unfortunately, this is only a guideline, not an established fact; as far as I know, the only rigorous argument in this vein is contained in D. Oberlin's paper \cite{Ob2}. There, an improvement to Kaufman's bound for projections is shown to imply an improvement over Wolff's bound for certain "toy" Furstenberg sets, which arise from a special, if natural, construction. So, even if the the Kaufman and Furstenberg problems are perhaps not equivalent for every $s \in (1/2,1)$, it is not outrageous exaggeration to claim that the former acts as a toy question towards the latter. 

The aim of this paper is to study Kaufman's bound for $\Hd K = 1$, and for any $1/2 < s < 1$ (the cases $s = 1/2$ and $0 \leq s < 1/2$ were solved by Bourgain \cite{Bo} and Oberlin \cite{Ob}, respectively). Here is the main result:

\begin{thm}\label{mainPacking} Given $1/2 < s < 1$, there exists a constant $\epsilon = \epsilon(s) > 0$ such that the following holds. If $K \subset \R^{2}$ be a Borel set with $\calH^{1}(K) > 0$, then
\begin{displaymath} \Hd(\{e \in S^{1} : \Pd \pi_{e}(K) \leq s\}) \leq s - \epsilon. \end{displaymath}
Here $\Pd$ stands for packing dimension. 
\end{thm}
Theorem \ref{mainPacking} does not improve over \eqref{kaufman}, since it only gives an upper bound for the dimension of $\{e : \Pd \pi_{e}(K) \leq s\}$ (a subset of $\{e : \Hd \pi_{e}(K) \leq s\}$). However, to the best of my knowledge, the bound "$s$" given by \eqref{kaufman} was, up to now, the best available even for $\Hd \{e : \Pd \pi_{e}(K) \leq s\}$. The assumption $\calH^{1}(K) > 0$ is a matter of convenience and could easily be relaxed to $\Hd K = 1$. As far as I can tell, the appearance of $\Pd$ is quite crucial for the proof strategy, and dealing with $\Hd$ requires a new idea. On the other hand, there is some hope that the proof strategy behind Theorem \ref{mainPacking} could give a an $\epsilon$-improvement over Wolff's bound for the upper box dimension of Furstenberg $s$-sets, $1/2 < s < 1$. This requires further investigation.

\subsection{Outline of the proof} In short, the proof of Theorem \ref{mainPacking} consists of two steps. One is to consider special sets $K$, which are roughly of the form $K = A \times B$, where $A$ is $s$-dimensional and $B$ is $(1 - s)$-dimensional. For such sets, one can prove Theorem \ref{mainPacking} by a direct argument, which uses tools from additive combinatorics (see Section \ref{productSets} below). The second step is to reduce the proof for general sets to the special case. This involves first pigeonholing a suitable scale $\delta > 0$ to work on. Then, one makes a counter assumption (namely: $K$ has an almost $s$-dimensional set of $s$-dimensional projections) both at scales $\delta^{1/2}$ and $\delta$. This evidently relies on having information about $\Pd \pi_{e}(K)$ and not just $\Hd \pi_{e}(K)$. If the counter assumption is strong enough, one can find the following structure inside $K$: there is a $(\delta^{1/2} \times 1)$-tube $T$ such that if one blows up $T \cap K$ into the unit square, then the resulting set $\tilde{K}$ "behaves like a $s \times (1 - s)$-dimensional product set with an almost $s$-dimensional set of $s$-dimensional projections". In effect, this means that the existence of $\tilde{K}$ contradicts the result obtained in the first part of the proof. Hence, one finally obtains a contradiction.

\section{Preliminaries and Kaufman's bound}

The purpose of this section is to record some preliminaries, notation and auxiliary results, and give a quick proof of the well-known and easy estimate $\Hd \{e : \Pd \pi_{e}(K) \leq s\} \leq s$ (that is, Theorem \ref{main} without the $\epsilon$-improvement).

First, I observe that in place of Theorem \ref{mainPacking}, it suffices to prove its analogue for \emph{upper box dimension}:
\begin{thm}\label{main} Given $1/2 < s < 1$, there exists a constant $\epsilon = \epsilon(s) > 0$ such that the following holds. If $K \subset \R^{2}$ be a compact set with $\calH^{1}(K) > 0$, then
\begin{displaymath} \Hd(\{e \in S^{1} : \Bd \pi_{e}(K) \leq s\}) \leq s - \epsilon. \end{displaymath} 
\end{thm}
Here $\Bd$ stands for the upper box (or Minkowski) dimension, which, for bounded sets $A \subset \R^{d}$, is defined by
\begin{displaymath} \Bd A := \limsup_{\delta \to 0} \frac{\log N(A,\delta)}{-\log \delta}. \end{displaymath}
The quantity $N(A,\delta)$ is the least number of balls of radius $\delta$ required to cover $A$. The fact that Theorem \ref{main} implies Theorem \ref{mainPacking} follows immediately from Lemma 4.5 in \cite{O2}.

The proof of Theorem \ref{main} proceeds by counter assumption and contradiction. Namely, I will assume that $\calH^{s - \epsilon_{0}/2}(\{e : \Bd \pi_{e}(K) \leq s\}) > 0$ for some (very small) $\epsilon_{0} > 0$. In particular, it follows that
\begin{displaymath} \{e \in S^{1} : N(\pi_{e}(K),\delta) \leq \delta^{-s - \epsilon_{0}/2} \text{ for all } 0 < \delta \leq \delta_{0}\} \end{displaymath}
has positive $(s - \epsilon_{0}/2)$-dimensional measure for small enough $\delta_{0} > 0$. Replacing $s$ by $s - \epsilon_{0}/2$ for notational convenience, I will assume that $\calH^{s}(E) > 0$, where
\begin{equation}\label{counterAss} E := \{e \in S^{1} : N(\pi_{e}(K),\delta) \leq \delta^{-s - \epsilon_{0}} \text{ for all } 0 < \delta \leq \delta_{0}\}. \end{equation}

Throughout the paper, I will use four types of "less than" inequality signs: $\leq$, $\lesssim$, $\lesssim_{\log}$ and $\lessapprox$. The first is most likely familiar to the reader, while $A \lesssim B$ means that there exists a constant $C \geq 1$ such that $A \leq CB$. If the dependence of $C$ on some parameter $p$ should be emphasised, this will be denoted by $A \lesssim_{p} B$. The inequality sign $A \lesssim_{\log} B$ means that
\begin{displaymath} A \lesssim \log^{C}(1/\delta) B, \end{displaymath}
where $C \geq 1$ is some constant (always quite small, $C \leq 10$), and $\delta > 0$ is a \emph{scale}, whose meaning will be clear later. Finally, the notation $A \lessapprox B$ means that
\begin{displaymath} A \leq C_{\epsilon_{0}}\delta^{C\epsilon_{0}}B. \end{displaymath}
Here $\epsilon_{0}$ is the "counter assumption parameter" from \eqref{counterAss}, $C_{\epsilon_{0}} \geq 1$ is a constant depending only on $\epsilon_{0}$ and "harmless parameters", and $C \geq 1$ is a constant depending only on "harmless parameters". These "harmless parameters" consist of quantities, which are regarded as "fixed" throughout the proof; a typical example is the number $s$. The notations $A \geq/\gtrsim/\gtrsim_{\log}/\gtrapprox B$ mean that $B \leq/\lesssim/\lesssim_{\log}/\lessapprox A$, and the notations $A =/\sim/\sim_{\log}/\approx B$ stand for two-sided inequalities.

If a little imprecision is allowed for a moment, the entire proof of Theorem \ref{main} will consist of a finite chain of inequalities of the form $A_{1} \lessapprox A_{2} \lessapprox \ldots \lessapprox A_{m}$, and finally the observation that $A_{1} \gtrsim \delta^{-\epsilon_{1}}A_{m}$ for some absolute constant $\epsilon_{1} > 0$. Thus, if $\delta,\epsilon_{0} > 0$ are small enough, a contradiction is reached.

The next definition contains a $\delta$-discretised analogue of "positive $t$-dimensional measure":

\begin{definition}[$(\delta,t)$-sets]\label{deltaTSet} Fix $\delta,t > 0$. A finite $\delta$-separated set $P \subset \R^{d}$ is called a $(\delta,t)$-set, if
\begin{equation}\label{deltaTSet} |P \cap B(x,r)| \lesssim_{\log} \left(\frac{r}{\delta} \right)^{t} \end{equation}
for all $x \in \R^{d}$ and $\delta \leq r \leq 1$. Here and below, $|\cdot |$ stands for cardinality. The set $P$ is called a \emph{generalised $(\delta,t)$-set}, if it satisfies the following relaxed version of \eqref{deltaTSet}:
\begin{displaymath} |P \cap B(x,r)| \lessapprox \left(\frac{r}{\delta} \right)^{t} \end{displaymath}
for all $x \in \R^{d}$ and $\delta \leq r \leq 1$.
\end{definition}

The definition of \emph{generalised $(\delta,t)$-set} is slightly vague, and the meaning will be best clarified in actual use below. In the proofs, a typical application is the following: a certain $\delta$-separated set $P$ is found, and one observes that the bound $|P \cap B(x,r)| \leq C_{\epsilon_{0}}\delta^{-C\epsilon_{0}}(r/\delta)^{t}$ holds for all $r \geq \delta$, and some constants $C,C_{\epsilon_{0}} \geq 1$. Then, Definition \ref{deltaTSet} allows me to call $P$ a (generalised) $(\delta,t)$-set without cumbersome book-keeping of the constants $C,C_{\epsilon_{0}}$.

The rationale behind the definition of $(\delta,s)$-sets is the fact that large $(\delta,s)$-sets can be found, for  any $\delta > 0$, inside a set with positive $s$-dimensional Hausdorff content. The following proposition is Proposition A.1 in \cite{FO} (the result in \cite{FO} is stated in $\R^{3}$, but the verbatim same proof works in every dimension):
\begin{proposition}\label{deltasSet} Let $\delta > 0$, and let $B \subset \R^{2}$ be a set with $\calH^{s}_{\infty}(B) =: \kappa > 0$. Then, there exists a $(\delta,s)$-set $P \subset B$ with cardinality $|P| \gtrsim \kappa \cdot \delta^{-s}$. In fact, the $(\delta,s)$-set property \eqref{deltaTSet} even holds with "$\lesssim$" instead of "$\lesssim_{\log}$" for $P$. \end{proposition}

Now, as a warm-up for things to come, but also for real use, I present a quick proof of the easy bound $\Hd \{e : \Pd \pi_{e}(K)\} \leq s$. As with Theorem \ref{mainPacking}, it suffices to prove that $\Hd \{e : \Bd \pi_{e}(K) \leq s\}$, and this follows almost immediately from the next proposition:
\begin{proposition}\label{kaufmanProp} Fix $\delta > 0$. Let $0 < s < 1$ and let $K \subset B(0,1) \subset \R^{2}$ be a set with $\calH^{1}_{\infty}(K) \sim 1$. Assume that $E \subset S^{1}$ is a $\delta$-separated set with $|E| \gtrsim \delta^{-s}$. Then, there exists a vector $e \in E$ with $N(\pi_{e}(K),\delta) \gtrsim_{\log} \delta^{-s}$. 
\end{proposition}

\begin{proof} By Proposition \ref{deltasSet}, there exists a $(\delta,1)$-set $P \subset K$ with $|P| \sim \delta^{-1}$. It suffices to find $e \in E$ such that $N(\pi_{e}(P),\delta) \gtrsim_{\log} \delta^{-s}$. Let $E_{0} \subset E$ be the set of vectors $e \in E$ such that the claim fails: more precisely, $N(\pi_{e}(P),\delta) \leq M$ for $M = c\delta^{-t}/\log^{C}(1/\delta)$, where $c,C > 0$ are suitable constants. It suffices to show that $|E_{0}| < |E|$, if $c > 0$ is small enough. Fix $e \in E_{0}$. Then, it is easy to check using Cauchy-Schwarz (or see the proof of Proposition 4.10 in \cite{O2}) that there exist $\gtrsim |P|^{2}/M \gtrsim \delta^{-2}/M$ pairs $(p_{1},p_{2}) \in P \times P$ such that
\begin{displaymath} |\pi_{e}(p_{1}) - \pi_{e}(p_{2})| \leq \delta. \end{displaymath}
Since the lower bound $|P|^{2}/M$ is far greater than $|P|$ for small enough $\delta > 0$, one in fact has the same lower bound for pairs $(p_{1},p_{2})$ satisfying additionally $p_{1} \neq p_{2}$. Consequently,
\begin{displaymath} \sum_{e \in E_{0}} |\{(p_{1},p_{2}) \in P \times P : p_{1} \neq p_{2} \text{ and } |\pi_{e}(p_{1}) - \pi_{2}(p_{2})| \leq \delta\}| \gtrsim \frac{|E_{0}|}{\delta^{2}M}. \end{displaymath}
On the other hand, using the geometric fact that $\{e \in S^{1} : |\pi_{e}(p_{1}) - \pi_{e}(p_{2})| \leq \delta\}$ is the union of two arcs of length $\lesssim \delta/|p_{1} - p_{2}|$, one has
\begin{align*} \sum_{e \in E} & |\{(p_{1},p_{2}) \in P \times P : p_{1} \neq p_{2} \text{ and } |\pi_{e}(p_{1}) - \pi_{e}(p_{2})| \leq \delta\}|\\
& = \sum_{p_{1} \neq p_{2}} |\{e \in E : |\pi_{e}(p_{1}) - \pi_{e}(p_{2})| \leq \delta\}|\\
& \lesssim \sum_{p_{1} \neq p_{2}} \frac{1}{|p_{1} - p_{2}|} \lesssim \sum_{p_{1}} \sum_{\delta \leq 2^{j} \leq 1} 2^{-j} |P \cap B(p_{1},2^{j})|\\
& \lesssim_{\log} \sum_{p_{1}} \sum_{\delta \leq 2^{j} \leq 1} 2^{-j} \cdot \frac{2^{j}}{\delta} \lesssim_{\log} \delta^{-2}. \end{align*}
Comparing the lower and upper bounds leads to
\begin{displaymath} |E_{0}| \lesssim_{\log} M = \frac{c\delta^{-t}}{\log^{C}(1/\delta)}. \end{displaymath}
For $c > 0$ sufficiently small and $C \geq 1$ sufficiently large, this gives $|E_{0}| < |E|$, and the proof is complete.
\end{proof}

\begin{cor} If $\calH^{1}(K) > 0$, then $\Hd \{e \in S^{1} : \Bd \pi_{e}(K) \leq s\} \leq s$. \end{cor}
\begin{proof} If the statement fails, then $\calH^{s + 2\epsilon}(\{e : \Bd \pi_{e}(K) \leq s\}) > 0$ for some $\epsilon > 0$. By definition of $E_{s}(K)$, this implies that the set
\begin{displaymath} E_{i} := \{e : N(\pi_{e}(K),\delta) \leq \delta^{-s - \epsilon} \text{ for all } \delta \leq 1/i\} \end{displaymath}
has positive $(s + 2\epsilon)$-dimensional measure for some $i \in \N$. In particular, $E_{i}$ contains a $\delta$-separated set of cardinality $\gtrsim \delta^{-s - 2\epsilon}$ for all $\delta \leq 1/i$. For small enough $\delta > 0$, this violates Proposition \ref{kaufmanProp}. \end{proof}

\begin{remark}\label{boxDimE} Proposition \ref{kaufmanProp} is also crucial for the proof of the main theorem. Recall the set $E$ in the main counter assumption \eqref{counterAss}. Then
\begin{equation}\label{boxE} N(E,\delta) \leq \delta^{-s - 2\epsilon_{0}} \end{equation}
for small enough $\delta \leq \delta_{0}$. Indeed, in the opposite case Proposition \ref{kaufmanProp} would imply that $N(\pi_{e}(K),\delta) \gtrsim_{\log} \delta^{-s - 2\epsilon_{0}}$ for some $e \in E$, violating the definition of $E$ for small enough $\delta > 0$. For simplicity and without loss of generality, I will assume that \eqref{boxE} holds for all $0 < \delta \leq \delta_{0}$. \end{remark}

\section{Product-like sets}\label{productSets}

The main result of this section is a technical statement, Proposition \ref{productProp}, about "product-like" sets, which will be useful later on in the context of general sets. A simple qualitative corollary of Proposition \ref{productProp} would state the following. Assume that $A \subset \R$ is $s$-dimensional and $B \subset \R$ is $\tau$-dimensional, $\tau > 0$. Then, for any $s$-dimensional set $E \subset S^{1}$, there exists $e \in E$ such that 
\begin{displaymath} \Bd \pi_{e}(A \times B) \geq s + \epsilon, \end{displaymath}
where $\epsilon > 0$ only depends on $s$ and $\tau$. Here is the quantitative version:

\begin{proposition}\label{productProp} Let $0 < s < 1$, $\epsilon,\tau > 0$. Let $B \subset [0,1]$ be a $(\delta,\tau)$-set of cardinality $|B| \gtrsim \delta^{-\tau + \epsilon}$, and let $E \subset S^{1}$ be a $(\delta,s)$-set of cardinality $|E| \gtrsim \delta^{-s + \epsilon}$. For each $b \in B$, assume that $A_{b} \subset [0,1]$ is a $(\delta,s)$-set of cardinality $|A_{b}| \gtrsim\delta^{-s + \epsilon}$, and let $P$ be the $(\delta,s + \tau)$-set
\begin{equation}\label{P} P := \bigcup_{b \in B} A_{b} \times \{b\}. \end{equation}
Then, if $\epsilon$ is small enough (depending only on $\tau,s$), then
\begin{equation}\label{form23} N(\pi_{e}(P),\delta) \geq \delta^{-s - \epsilon} \quad \text{ for some } e \in E, \end{equation}
for all sufficiently small $\delta > 0$ (depending only on $s,\tau$, and the implicit constants behind the $\sim$ notation). 
\end{proposition}

\begin{remark} In this section, Section \ref{productSets}, the constant $\epsilon_{0}$ from the main counter assumption \eqref{counterAss} does not make an appearance. So, it will cause no confusion, if the notations $\lessapprox$, $\gtrapprox$ and $\approx$ are temporarily re-purposed for the needs of Proposition \ref{productProp}. In particular, the failure of \eqref{form23} will be denoted by $N(\pi_{e}(P),\delta) \lessapprox \delta^{-s}$, as in \eqref{form1} below. Similarly, the cardinality of $B$ is $|B| \approx \delta^{-\tau}$ and so on. 
\end{remark}

Before starting the proof, I recall two standard results from additive combinatorics. The first is the \emph{Balog-Szemer\'edi-Gowers theorem}. The statement below is taken verbatim from p. 196 in \cite{Bo}. For a proof, see \cite{TV}, p. 267.
\begin{thm}[Balog-Szemer\'edi-Gowers]\label{BSG} There exists an absolute constant $C \geq 1$ such that the following holds. Let $A,B \subset \R$ be finite sets, and assume that $G \subset A \times B$ is a set of pairs such that
\begin{displaymath} |G| \geq \frac{|A||B|}{K} \quad \text{and} \quad |\{x + y : (x,y) \in G\}| \leq K|A|^{1/2}|B|^{1/2} \end{displaymath}
for some $K > 1$. Then, there exist $A' \subset A$ and $B' \subset B$ satisfying
\begin{itemize}
\item $|A'| \geq K^{-C}|A|$, $|B'| \geq K^{-C}|B|$,
\item $|A' + B'| \leq K^{C}|A|^{1/2}|B|^{1/2}$, and
\item $|G \cap (A' \times B')| \geq K^{-C}|A||B|$. 
\end{itemize}
\end{thm}

The second auxiliary result is the Pl\"unnecke-Ruzsa inequality, whose proof can also be found in \cite{TV}:
\begin{thm}[Pl\"unnecke-Ruzsa]\label{PR} Assume that $A,B \subset \R$ are finite sets such that
\begin{displaymath} |A + B| \leq C|A| \end{displaymath}
for some integer $C \geq 1$. Then
\begin{displaymath} |B^{m} \pm B^{n}| \leq C^{m + n}|A| \end{displaymath}
for all $m,n \in \N$.
\end{thm}

\begin{remark} Theorem \ref{PR} will be applied in the following form: if $A,B \subset \R$ are $\delta$-separated sets with $|A| \approx |B|$ and
\begin{displaymath} N(A + B,\delta) \lessapprox |A|, \end{displaymath}
then $N(B + B,\delta) \lessapprox |A|$. This statement follows easily from Theorem \ref{PR} by considering the sets $[A]_{\delta} = \{[a]_{\delta} : a \in A\} \subset \delta \Z$ and $[B]_{\delta} := \{[b]_{\delta} ; b \in B\} \subset \delta \Z$, where $[x]_{\delta} \in \delta \Z$ stands for the largest number $\delta n \in \delta \Z$ satisfying $\delta n \leq x$. Then the hypothesis $N(A + B,\delta) \lessapprox |A|$ implies that $|[A]_{\delta} + [B]_{\delta}| \lessapprox |[A]|_{\delta}$, so Theorem \ref{PR} can be applied.
\end{remark}

\begin{proof}[Proof of Proposition \ref{productProp}] For later technical convenience, I will already make the assumption that all the vectors in $E$ are "roughly horizontal", which precisely means the following: $e^{1} \sim 1$ for all $(e^{1},e^{2}) \in E$, and a $\delta$-tube perpendicular to any one of the vectors $e \in E$ contains at most one point of the form $(a,b) \in A_{b} \times \{b\}$ for any fixed $b \in B$. This can be arranged by replacing $E$ and the sets $A_{b}$ by suitable subsets $A_{b}' \subset A_{b}$ and $E' \subset E$ with $|A_{b}'| \sim |A|$ and $|E'| \sim |E|$.

Another convenient extra hypothesis is that $A_{b} \subset \delta \Z$ for all $b \in B$. This can be achieved by replacing the sets $A_{b}$ by the sets $[A_{b}]_{\delta} := \{[a]_{\delta} : a \in A_{b}\}$. Neither the hypotheses nor the conclusion of the theorem are relevantly affected by the passage from $A_{b}$ to $[A_{b}]_{\delta}$, since the sets $A_{b}$ were assumed to be $\delta$-separated to begin with. 

The proof can now start in earnest. I make the counter assumption that
\begin{equation}\label{form1} N(\pi_{e}(P),\delta) \lessapprox \delta^{-s}, \qquad e \in E, \end{equation}
and gradually work towards a contradiction. Fix a vector $e_{0} = (e_{0}^{1},e_{0}^{2}) \in E$, and write $A := \pi_{e_{0}}(P)$, so that 
\begin{equation}\label{form109} N(A,\delta) \approx \delta^{-s} \end{equation}
by \eqref{form1}. I will first argue that one may assume $e_{0} = (1,0)$ without loss of generality. Note that
\begin{displaymath} e_{0}^{1}A_{b} + e_{0}^{2}b \subset \pi_{e_{0}}(P) \subset A \end{displaymath}
for every $b \in B$. Now, if $A_{b}$ is replaced by $A_{b}' := e_{0}^{1}A_{b} + e_{0}^{2}b$, and $P'$ is built from these $A_{b}'$ as in \eqref{P}, then it is clear that $P'$ is of the form discussed in the statement of the theorem, and
\begin{displaymath} \pi_{(1,0)}(P') \subset A, \end{displaymath}
and \eqref{form1} holds for $P'$ and the vectors $E' := \{(e^{1}/e_{0}^{1},e^{2} - e^{1}[e^{2}_{0}/e^{1}_{0}]) : (e^{1},e^{2}) \in E\}$. Since $E'$ is obtained from $E$ by an affine transformation of determinant $1/e_{0}^{1} \sim 1$, one sees that $E'$ is a $(\delta,s)$-set of cardinality $\approx \delta^{-s}$. Thus, one can first prove the theorem for $P'$ instead of $P$, and finally do the affine transformation in the other direction to get the result for $P$. So, assume without loss of generality that $\pi_{(1,0)}(P) \subset A$, which implies that
\begin{equation}\label{A} A_{b} \subset A, \qquad b \in B. \end{equation}
Finally, since one was also allowed to assume $A_{b} \subset \delta \Z$, it follows from \eqref{form109} that
\begin{equation}\label{form10} A \subset \delta \Z \quad \text{and} \quad |A| \approx \delta^{-s}. \end{equation}

 For each $e \in E$, cover $P$ by $\lessapprox \delta^{-s}$ tubes of dimensions $\delta \times 10$, perpendicular to $e$. Denote these tubes by $\calT_{e}$, and write
\begin{displaymath} \calT := \bigcup_{e \in E} \calT_{e}, \end{displaymath}
so that $|\calT| \lessapprox \delta^{-2s}$. Then, for $e \in E$ and distinct $p,q \in P$, write $p \sim_{e} q$, if there exists $T \in \calT_{e}$ such that $p,q \in T$. Further, define $p \sim q$, if $p \sim_{e} q$ for some $e \in E$ (that is, $p,q \in T$ for some $T \in \calT$). The first task is to find a lower bound for the number of pairs
\begin{displaymath} Q := \{(p,q) \in P \times P : p \sim q\}. \end{displaymath} 
The desired estimate is $|Q| \gtrapprox \delta^{-2s - 2\tau} \approx |P|^{2}$. To this end, note that for fixed $e \in E$, it is easy to check (using Cauchy-Schwarz) that
\begin{displaymath} |\{(p,q) \in P \times P : p \sim_{e} q\}| \gtrapprox \delta^{-s - 2\tau}, \end{displaymath}
so that
\begin{equation}\label{form2} \sum_{e \in E} |\{(p,q) \in P \times P : p \sim_{e} q\}| \gtrapprox \delta^{-2s - 2\tau} \approx |P|^{2}. \end{equation}
This almost looks like the desired estimate, but the sets in the summation need not be disjoint for distinct $e \in E$. However, using the $(\delta,s)$-set property of $E$ (and the geometry of $\{e \in S^{1} : p \sim_{e} q\}$), the left hand side of \eqref{form2} can be estimated from above as follows:
\begin{align*} \textup{L.H.S of \eqref{form2} } & = \sum_{(p,q) \in Q} |\{e \in E : p \sim_{e} q\}| \lesssim \sum_{(p,q) \in Q} \frac{1}{|p - q|^{s}}\\
& \leq |Q|^{1/q} \Bigg(\sum_{p \neq q} \frac{1}{|p - q|^{s + \tau}} \Bigg)^{1/p} \lesssim_{\log} |Q|^{1/q}|P|^{2/p}, \end{align*} 
where $p > 1$ is chosen so that $ps = s + \tau$, and the last inequality follows from the fact that $P$ is a $(\delta,s + \tau)$-set of cardinality $\approx \delta^{-s - \tau}$. It follows from this and \eqref{form2} that
\begin{equation}\label{form3} |Q| \gtrapprox |P|^{2}, \end{equation}
as claimed. Further, note that
\begin{equation}\label{form4} \sum_{b_{1} \neq b_{2}} |\{(p,q) \in A^{b_{1}} \times A^{b_{2}} : p \sim q\}| = |Q| \gtrapprox |P|^{2}, \end{equation} 
where $A^{b_{i}} := A_{b_{i}} \times \{b_{i}\} \subset P$. This follows from \eqref{form3} and the fact that the tubes in $\calT$ are fairly vertical (so that there are no relations $p \sim q$ with $p,q \in A^{b}$).

Fixing $b_{1},b_{2}$, let $\calT_{b_{1},b_{2}} \subset \calT$ be a collection of tubes such that every pair $(p,q) \in A^{b_{1}} \times A^{b_{2}}$ with $p \sim q$ is contained in a tube in $\calT_{b_{1},b_{2}}$. Such tubes exist by definition of the relation "$\sim$", but they need not be unique: pick exactly one tube $T_{(p,q)}$ for every pair $(p,q) \in A^{b_{1}} \times A^{b_{2}}$. Then
\begin{displaymath} |\calT_{b_{1},b_{2}}| \geq |\{(p,q) \in A^{b_{1}} \times A^{b_{2}} : p \sim q\}|, \end{displaymath}
because the mapping $(p,q) \mapsto T_{(p,q)}$ is injective by the assumption that the vectors $e$ are "roughly horizontal" (see the first paragraph of the proof for a precise statement). Consequently, by \eqref{form4},
\begin{equation}\label{form105} \sum_{b_{1}, b_{2}} |\calT_{b_{1},b_{2}}| \gtrapprox |P|^{2}. \end{equation}
(Here $\calT_{b,b} := \emptyset$ for $b \in B$.) Since $\calT_{b_{1},b_{2}} \subset \calT$, and $|\calT| \lessapprox \delta^{-2s}$, one sees from \eqref{form105} that $|\calT_{b_{1},b_{2}}| \approx \delta^{-2s}$ for "most" pairs $(b_{1},b_{2}) \in B^{2}$. In fact, something slightly better is needed, and follows from the next Cauchy-Schwarz estimate, and $|\calT| \lessapprox \delta^{-2s}$:
\begin{align*} \sum_{b_{1},b_{2},b_{3}} |\calT_{b_{1},b_{2}} \cap \calT_{b_{2},b_{3}}| & = \sum_{T \in \calT} \sum_{b_{2}} \sum_{b_{1}, b_{3}} \chi_{\calT_{b_{1},b_{2}}}(T)\chi_{\calT_{b_{2},b_{3}}}(T)\\
& = \sum_{T \in \calT} \sum_{b_{2}} \left(\sum_{b} \chi_{\calT_{b,b_{2}}}(T) \right)^{2}\\
& \geq \frac{1}{|\calT||B|} \left( \sum_{T \in \calT} \sum_{b,b_{2}} \chi_{\calT_{b,b_{2}}}(T) \right)^{2}\\
& \gtrapprox \frac{|P|^{4}}{|\calT||B|} \gtrapprox \delta^{-2s}|B|^{3}. \end{align*} 
Since $|\calT_{b_{1},b_{2}} \cap \calT_{b_{2},b_{3}}| \lessapprox \delta^{-2s}$ for any triple $(b_{1},b_{2},b_{3})$, it follows that there exist $\approx |B|^{3}$ triples $(b_{1},b_{2},b_{3})$ with the property that $|\calT_{b_{1},b_{2}} \cap \calT_{b_{2},b_{3}}| \approx \delta^{-2s}$. As will be made precise in a moment, the condition $|\calT_{b_{1},b_{2}} \cap \calT_{b_{2},b_{3}}| \approx \delta^{-2s}$ roughly means that there are $\approx \delta^{-2s}$ points in $A_{b_{1}} \times A_{b_{2}}$ such that the projection of these points is small in a certain direction, determined by $b_{1},b_{2},b_{3}$. 

Consider a triple $(b_{1},b_{2},b_{3}) \in B^{3}$ with $|\calT_{b_{1},b_{2}} \cap \calT_{b_{2},b_{3}}| \approx \delta^{-2s}$. Fix a tube $T \in \calT_{b_{1},b_{2}} \cap \calT_{b_{2},b_{3}}$. Since $T \in \calT_{b_{1},b_{2}}$, one has $T = T_{(p_{1},q)}$ for some (unique) pair of points 
\begin{displaymath} p_{1} = (a_{1},b_{1}) \in A_{b_{1}} \times \{b_{1}\} \quad \text{and} \quad q = (a_{2},b_{2}) \in A_{b_{2}} \times \{b_{2}\}. \end{displaymath}
Similarly, because $T \in \calT_{b_{2},b_{3}}$, there exists yet another (unique) point 
\begin{displaymath} p_{3} = (a_{3},b_{3}) \in A_{b_{3}} \times \{b_{3}\} \end{displaymath}
such that $T = T_{(q,p_{3})}$. In particular, gathering all the pairs $(a_{1},a_{3}) \in A_{b_{1}} \times A_{b_{3}}$ obtained this way, one sees that the tubes $T \in \calT_{b_{1},b_{2}} \cap \calT_{b_{2},b_{3}}$ give rise to a subset 
\begin{displaymath} G_{b_{1},b_{2},b_{3}}' \subset A_{b_{1}} \times A_{b_{3}} \stackrel{\eqref{A}}{\subset} A \times A \end{displaymath}
of cardinality 
\begin{displaymath} |G_{b_{1},b_{2},b_{3}}'| = |\calT_{b_{1},b_{2}} \cap \calT_{b_{2},b_{3}}| \approx \delta^{-2s} \approx |A|^{2}. \end{displaymath}

From now on, restrict attention to triples $(b_{1},b_{2},b_{3}) \in B^{3}$ such that 
\begin{equation}\label{separation} \min_{i \neq j} |b_{i} - b_{j}| \approx 1. \end{equation}
Since the triples \textbf{failing} this condition have cardinality far less than $|B|^{3}$ (using the $(\delta,\tau)$-set hypothesis of $B$, and the assumption $|B| \approx \delta^{-\tau}$), one sees that $|\calT_{b_{1},b_{2}} \cap \calT_{b_{2},b_{3}}| \approx \delta^{-2s}$ holds for $\approx |B|^{3}$ triples satisfying \eqref{separation}. Fix one such triple, assume that $b_{1} < b_{3}$, and consider a pair $(a_{1},a_{3}) \in G_{b_{1},b_{2},b_{3}}'$. Recall how such points arise, and the notation for $p_{1},q,p_{3}$. Let 
\begin{displaymath} L = \left\{x = \frac{a_{3} - a_{1}}{b_{3} - b_{1}}y + \frac{a_{1}b_{3} - a_{3}b_{1}}{b_{3} - b_{1}} : y \in \R\right\} \end{displaymath}
be the line spanned by $p_{1}$ and $p_{3}$; then, since $p_{1},q,p_{3}$ all lie in the common $\delta$-tube $T$, the line $L$ passes at distance $\lesssim \delta$ from $q = (a_{2},b_{2}) \in A_{b_{2}} \times \{b_{2}\}$, which is equivalent to
\begin{displaymath} \left| \frac{a_{3}(b_{2} - b_{1}) + a_{1}(b_{3} - b_{2})}{b_{3} - b_{1}} - a_{2} \right| \lesssim \delta. \end{displaymath}
Recalling \eqref{separation}, this further implies that
\begin{displaymath} \left| \left(a_{1} + \frac{b_{2} - b_{1}}{b_{3} - b_{2}} a_{3} \right) - \frac{b_{3} - b_{1}}{b_{3} - b_{2}}a_{2} \right| \lessapprox \delta. \end{displaymath}
Consequently, if $\pi_{b_{1},b_{2},b_{3}}$ stands for the projection-like mapping
\begin{equation}\label{pi} \pi_{b_{1},b_{2},b_{3}}(x,y) = x + \frac{b_{2} - b_{1}}{b_{3} - b_{2}}y, \end{equation}
then 
\begin{displaymath} \dist \left(\pi_{b_{1},b_{2},b_{3}}(G_{b_{1},b_{2},b_{3}}'),\frac{b_{3} - b_{1}}{b_{3} - b_{2}}A_{b_{2}} \right) \lessapprox \delta. \end{displaymath}
Observing  $N([(b_{3} - b_{1})/(b_{3} - b_{2})]A_{b_{2}},\delta) \lessapprox \delta^{-s}$, it follows that
\begin{equation}\label{form11} N(\pi_{b_{1},b_{2},b_{3}}(G_{b_{1},b_{2},b_{3}}'),\delta) \lessapprox \delta^{-s} \approx |A|. \end{equation}

In fact, this holds for any triple $(b_{1},b_{2},b_{3}) \in B^{3}$ satisfying \eqref{separation} by definition of $G_{b_{1},b_{2},b_{3}}'$, but the information is most useful, if $|G_{b_{1},b_{2},b_{3}}'| \approx |A|^{2}$. Write
\begin{displaymath} F_{b_{1},b_{2},b_{3}} := \left\{\left(a_{1},\left[\frac{b_{2} - b_{1}}{b_{3} - b_{2}}a_{2}\right]_{\delta} \right) : (a_{1},a_{2}) \in G'_{b_{1},b_{2},b_{3}}\right\} \subset A \times \left[\frac{b_{2} - b_{1}}{b_{3} - b_{2}}A\right]_{\delta}. \end{displaymath}
It follows easily from \eqref{form11} (and recalling $A \subset \delta \Z$) that
\begin{displaymath} |\{f_{1} + f_{2} : (f_{1},f_{2}) \in F_{b_{1},b_{2},b_{3}} \}| \lessapprox |A|. \end{displaymath}
Moreover, since $|(b_{2} - b_{1})/(b_{3} - b_{2})| \approx 1$ for every triple $(b_{1},b_{2},b_{3})$ satisfying \eqref{separation}, it follows that $|F_{b_{1},b_{2},b_{3}}| \approx |A|^{2}$ whenever \eqref{separation} holds and $|G_{b_{1},b_{2},b_{3}}'| \approx |A|^{2}$. For such a \emph{good} triple $(b_{1},b_{2},b_{3})$, the Balog-Szemer\'edi-Gowers theorem, Theorem \ref{BSG}, implies that there exist subsets
\begin{displaymath} D^{1}_{b_{1},b_{2},b_{3}} \subset A \quad \text{and} \quad \tilde{D}^{2}_{b_{1},b_{2},b_{3}} \subset \left[\frac{b_{2} - b_{1}}{b_{3} - b_{2}}A\right]_{\delta} \end{displaymath}
such that $|D^{1}_{b_{1},b_{2},b_{3}}|,|\tilde{D}^{2}_{b_{1},b_{2},b_{3}}| \approx |A|$,
\begin{equation}\label{form17} |(D^{1}_{b_{1},b_{2},b_{3}} \times \tilde{D}^{2}_{b_{1},b_{2},b_{3}}) \cap F_{b_{1},b_{2},b_{3}}| \approx |A|^{2} \end{equation}
and
\begin{equation}\label{form118} |D^{1}_{b_{1},b_{2},b_{3}} + \tilde{D}^{2}_{b_{1},b_{2},b_{3}}| \lessapprox |A|. \end{equation}
Let
\begin{displaymath} D_{b_{1},b_{2},b_{3}}^{2} := \left\{a \in A : \left[\frac{b_{2} - b_{1}}{b_{3} - b_{2}}a \right]_{\delta} \in \tilde{D}^{2}_{b_{1},b_{2},b_{3}} \right\}. \end{displaymath}
It then follows from the definition of $F_{b_{1},b_{2},b_{3}}$ and \eqref{form17} that
\begin{equation}\label{form112} |G_{b_{1},b_{2},b_{3}}| := |(D^{1}_{b_{1},b_{2},b_{3}} \times D^{2}_{b_{1},b_{2},b_{3}}) \cap G_{b_{1},b_{2},b_{3}}'| \approx |A|^{2}. \end{equation}
for a good triple $(b_{1},b_{2},b_{3})$. Moreover, \eqref{form118} easily implies that
\begin{equation}\label{form13} N_{1} := N\left(D^{1}_{b_{1},b_{2},b_{3}} + \frac{b_{2} - b_{1}}{b_{3} - b_{2}}D^{2}_{b_{1},b_{2},b_{3}},\delta\right) \lessapprox |A|. \end{equation}
Finally, combining \eqref{form13} with the Pl\"unnecke-Ruzsa inequality, Theorem \ref{PR}, gives
\begin{equation}\label{form16} N_{2} := (D^{2}_{b_{1},b_{2},b_{3}} + D^{2}_{b_{1},b_{2},b_{3}},\delta) \lessapprox |A| \end{equation}
for any good triple $(b_{1},b_{2},b_{3})$. Since there are $\approx |B|^{3}$ good triples $(b_{1},b_{2},b_{3})$, one can find $b_{2},b_{3}$ such that \eqref{form112}--\eqref{form16} hold for $\approx |B|$ choices of $b_{1}$. Fix such $b_{2},b_{3} \in B$. Then, a simple Cauchy-Schwarz argument shows that $|G_{b_{1},b_{2},b_{3}} \cap G_{b_{1}',b_{2},b_{3}}| \approx |A|^{2}$ for $\approx |B|^{2}$ pairs $(b_{1},b_{1}')$, so that one can finally also fix $b_{1} \in B$ such that
\begin{equation}\label{form114} |G_{b}| := |G_{b_{1},b_{2},b_{3}} \cap G_{b,b_{2},b_{3}}| \approx |A|^{2} \end{equation}
for $\approx |B|$ choices of $b \in B$. For this specific (good triple) $(b_{1},b_{2},b_{3})$, I denote the set of $b \in B$ such that \eqref{form114} holds by $B_{0}$. With \eqref{form13} in mind, write 
\begin{displaymath} c_{b} := \frac{b_{2} - b}{b_{3} - b_{2}}, \qquad b \in B_{0}, \end{displaymath}
and abbreviate $c := c_{b_{1}}$ (note that $|c|,|c_{b}| \approx 1$ for all $b \in B_{0}$ by \eqref{separation}). Also, write
\begin{displaymath} D^{1} := D_{b_{1},b_{2},b_{3}}^{1}(\delta) \quad \text{and} \quad D^{2} := D^{2}_{b_{1},b_{2},b_{3}}(\delta), \end{displaymath}
where $R(\delta)$ stands for the $\delta$-neighbourhood of $R \subset \R^{d}$. To complete the proof, I repeat an argument of Bourgain (see p. 219 in \cite{Bo}). Assume for a moment that $x \in cD^{2} \times D^{2} \subset \R^{2}$ and $b \in B_{0}$. Then $\chi_{-G_{b}(\delta) - y}(x) = 1$, whenever 
\begin{displaymath} y \in -G_{b}(\delta) - x \subset -(D^{1} \times D^{2}) - (cD^{2} \times D^{2}) = -(D^{1} + cD^{2}) \times -(D^{2} + D^{2}), \end{displaymath}
(the first inclusion uses \eqref{form112} and \eqref{form114}) and the Lebesgue measure of such choices $y$ is evidently $\calL^{2}(G_{b}(\delta))$. This gives the inequality
\begin{displaymath} \chi_{cD^{2} \times D^{2}} \leq \frac{1}{\calL^{2}(G_{b}(\delta))} \int_{-(D^{1} + cD^{2}) \times -(D^{2} + D^{2})} \chi_{-G_{b} - y}\, dy \end{displaymath} 
which easily implies
\begin{displaymath} \chi_{cD^{2} + c_{b}D^{2}} \leq \frac{1}{\calL^{2}(G_{b}(\delta))} \int_{-(D^{1} + cD^{2}) \times -(D^{2} + D^{2})} \chi_{\pi_{b,b_{2},b_{3}}(-G_{b}) - \pi_{b,b_{2},b_{3}}(y)} \, dy, \quad b \in B_{0}, \end{displaymath}
by the definition of $\pi_{b,b_{2},b_{3}}$ (see \eqref{pi}). Finally, integrating the previous inequality and recalling \eqref{form13}, \eqref{form16} and \eqref{form11}, one obtains
\begin{equation}\label{conclusion} \calL^{1}(cD^{2} + c_{b}D^{2}) \lesssim \frac{(N_{1}\delta)(N_{2}\delta)}{\calL^{2}(G_{b}(\delta))} \calL^{1}(\pi_{b_{1},b,b_{2}}(G_{b})) \lessapprox \delta^{1 - s} \approx \delta |A|, \quad b \in B_{0}. \end{equation}
However, $cD^{2} \times D^{2}$ is the $\delta$-neighbourhood of a generalised $(\delta,2s)$-set in the plane, so Bourgain's discretized projection theorem, Theorem 5 in \cite{Bo}, can be applied with $\alpha := 2s < 2 =: d$ and any $\kappa > 0$. If $\mu_{1}$ is the natural probability measure on the $\delta$-neighbourhood of $\{c_{b} : b \in B_{0}\}$, then $\mu_{1}$ satisfies assumption (0.14) from \cite{Bo} for any $\tau_{0} > 0$ (recall the definition of the numbers $c_{b}$, in particular $|b_{2} - b_{3}| \approx 1$, recall that $B_{0} \subset B$ has cardinality $|B_{0}| \approx |B|$, and $B$ is a $(\delta,\tau)$-set). The conclusion (in (0.19) of \cite{Bo}) is that some $b \in B_{0}$ should violate \eqref{conclusion}. Thus, a contradiction is reached, and the proof is complete. 
\end{proof}

\section{General sets}

So far, the the scale $\delta > 0$ has been small but otherwise arbitrary. To prove Theorem \ref{main}, one needs to deal with a set $K \subset B(0,1)$ with $\calH^{1}(K) > 0$. To extract useful information from the main counter assumption \eqref{counterAss}, namely that
\begin{equation}\label{form24} N(\pi_{e}(K),\delta) \leq \delta^{-s - \epsilon_{0}}, \qquad 0 < \delta \leq \delta_{0}, \end{equation}
for all $e \in E$ with $\calH^{s}(E) > 0$, I will need a special scale $\delta > 0$ with the properties that $K$ looks approximately $1$-dimensional (in a rather weak sense) both at scales $\delta^{1/2}$ and $\delta$. Such a scale can be found with a pigeonholing argument, given in the first subsection below. Then, since the counter assumption concerns all (small) scales $\delta > 0$, it applies in particular to the specific scale the pigeon helped to find.

\subsubsection{Choosing the scale $\delta$} By choosing a subset of $K$, one may assume that $0 < \calH^{1}(K) < \infty$. I treat $\calH^{1}(K)$ as an absolute constant, so that $\calH^{1}(K) \sim 1$. Let $\mu$ be a Frostman measure supported on $K$, that is, $\mu(K) = 1$ and $\mu(B(x,r)) \lesssim r$ for all balls $B(x,r) \subset \R^{2}$. Next, let $\calB$ be an efficient $\delta_{0}$-cover for $K$, that is,
\begin{equation}\label{form19} \sup\{\diam B : B \in \calB\} \leq \delta_{0} \quad \text{and} \quad \sum_{B \in \calB} \diam(B) \lesssim \calH^{1}(K) \sim 1. \end{equation} 
For $j \in \N$ such that $2^{-j} \leq \delta_{0}$, set $\calB_{j} := \{B \in \calB : \diam(B) \sim 2^{-j}\}$, and observe that
\begin{displaymath} \sum_{2^{-j} \leq \delta_{0}} \sum_{B \in \calB_{j}} \mu(B) \geq \mu(K) = 1. \end{displaymath}
In particular, there exists an index $j \in \N$ with $2^{-j} \leq \delta_{0}$ and
\begin{equation}\label{form117} \sum_{B \in \calB_{j}} \mu(B) \gtrsim \frac{1}{(j - j_{0} + 1)^{2}}. \end{equation}
Here $j_{0} \in \N$ satisfies $2^{-j_{0}} \sim \delta_{0}$. Now, I declare that
\begin{displaymath} \delta := 2^{-2j}, \end{displaymath}
so that $\delta^{1/2} = 2^{-j}$. In particular, \eqref{form117} implies that
\begin{equation}\label{form18} \sum_{B \in \calB_{j}} \mu(B) \gtrsim_{\log} 1. \end{equation}
Observe that $|\calB_{j}| \lesssim \delta^{-1/2}$ by \eqref{form19}, and on the other hand every ball $B \in \calB_{j}$ satisfies $\mu(B) \lesssim \delta^{1/2}$. Thus, \eqref{form18} implies that there are $\sim_{\log} \delta^{-1/2}$ balls in $\calB_{j}$, denoted by $\calB_{j}^{G}$, such that
\begin{equation}\label{form20} \mu(B) \gtrsim_{\log} \delta^{1/2}, \qquad B \in \calB_{j}^{G}. \end{equation}
Discarding a few balls if necessary, one may assume that the 
\begin{equation}\label{form21} \dist(B,B') \geq \delta^{-1/2}, \qquad B,B' \in \calB_{j}^{G}. \end{equation}
For each ball $B \in B_{j}^{G}$, choose a $(\delta,1)$-set $P_{B} \subset B$ with $|P_{B}| \gtrsim_{\log} \delta^{-1/2}$. This is possible by Proposition \ref{deltasSet}, since \eqref{form20} and the linear growth of $\mu$ imply that $\calH_{\infty}^{1}(B \cap K) \gtrsim_{\log} \delta^{1/2}$.

For each ball $B \in \calB_{j}^{G}$, pick a single point $p_{B} \in P_{B}$, and write $P_{\delta^{1/2}} := \{p_{B} : B \in \calB_{j}^{G}\}$. Then $|P_{\delta^{1/2}}| \sim_{\log} \delta^{-1/2}$. Also, write
\begin{displaymath} P := P_{\delta} := \bigcup_{B \in \calB_{j}^{G}} P_{B}. \end{displaymath}
Then $|P_{\delta}| \sim_{\log} \delta^{-1}$, and $P_{B} = B \cap P$. I conclude the section by verifying that $P_{\delta^{1/2}}$ is a $(\delta^{1/2},1)$-set, and $P$ is a $(\delta,1)$-set. Towards the first claim, fix $x \in \R^{2}$ and $\delta^{1/2} \leq r \leq 1$. Then, writing $M := |B(x,r) \cap P_{\delta^{1/2}}|$, observe that
\begin{displaymath} r \gtrsim \mu(B(x,2r)) \geq \sum_{p_{B} \in B(x,r) \cap P_{\delta^{1/2}}} \mu(B) \gtrsim_{\log} M\delta^{1/2} \end{displaymath}
by \eqref{form20}. This gives $M \lesssim_{\log} r/\delta^{1/2}$, as desired. Next, consider the claim for $P$. For $\delta \leq r \leq \delta^{1/2}$, note that
\begin{displaymath} |B(x,r) \cap P| = |B(x,r) \cap P_{B}| \lesssim_{\log} \frac{r}{\delta} \end{displaymath}
by \eqref{form21} and the fact that $P_{B}$ is a $(\delta,1)$-set. Finally, for $\delta^{1/2} \leq r \leq 1$, observe that 
\begin{displaymath} \frac{r}{\delta} \gtrsim_{\log} |P_{\delta^{1/2}} \cap B(x,2r)| \cdot \delta^{-1/2} \gtrsim_{\log} |P \cap B(x,r)|, \end{displaymath}
since for every point in $p \in P \cap B(x,r)$, one has $p \in P_{B}$ for a certain $B \in \calB_{j}^{G}$, and then $p_{B} \in B(x,2r) \cap P_{\delta^{1/2}}$.  

I recap the achievements so far. For a certain scale $\delta \leq (\delta_{0})^{2}$, the following hold:
\begin{itemize}
\item $P \subset K$ is a $(\delta,1)$-set of cardinality $|P| \sim_{\log} \delta^{-1}$.
\item $P$ can be covered by $\sim_{\log} \delta^{-1/2}$ balls of diameter $\sim \delta^{1/2}$ in the collection $\calB_{j}^{G}$, which I will henceforth denote simply by $\calB$. For every $B \in \calB$, the set $P$ contains a special point $p_{B}$, and the set $P_{\delta^{1/2}} \subset P$ of these special points is a $(\delta^{1/2},1)$-set of cardinality $|P_{\delta^{1/2}}| \sim_{\log} \delta^{-1/2}$. 
\item Since $P_{\delta^{1/2}} \subset P \subset K$ and $\delta^{1/2} \leq \delta_{0}$, the main counter assumption \eqref{form24} implies that
\begin{equation}\label{form5} N(\pi_{e}(P_{\delta^{1/2}}),\delta^{1/2}) \leq \delta^{-(s + \epsilon_{0})/2} \end{equation}
and
\begin{equation}\label{form22} N(\pi_{e}(P),\delta) \leq \delta^{-s - \epsilon_{0}} \end{equation}
for $e \in E$. 
\end{itemize}

\subsubsection{The sets $E$ and $E_{\delta^{1/2}}$}\label{Edelta} Recall from Remark \ref{boxDimE} that
\begin{displaymath} N(E,\delta') \leq (\delta')^{-s - 2\epsilon_{0}}, \qquad \delta' \leq \delta_{0}. \end{displaymath}
This will presently be applied with $\delta' = \delta^{1/2}$, where $\delta > 0$ is the fixed scale from the discussion above. Since $\calH^{s}(E) > 0$, one can find (by Proposition \ref{deltasSet}) a $(\delta,s)$-subset of cardinality $\sim \delta^{-s}$. This finite subset will henceforth be denoted by $E$; note that \eqref{form5} and \eqref{form22} remain trivially valid. Since $N(E,\delta^{1/2}) \leq \delta^{-s/2 - \epsilon_{0}}$, and every arc of length $\delta^{1/2}$ can only contain $\lesssim \delta^{-s/2}$ points in $E$, it follows that there exist at least $\gtrsim \delta^{-s/2}$ arcs $J_{1},\ldots,J_{N} \subset S^{1}$ of length $\delta^{1/2}/10$ with
\begin{displaymath} |E \cap J_{i}| \gtrsim \delta^{-s/2 + \epsilon_{0}} \gtrapprox \delta^{-s/2}. \end{displaymath}
For every arc $J_{i}$, pick a single point, and denote the set thus obtained by $E_{\delta^{1/2}}$. By discarding a few points, one may assume that $E_{\delta^{1/2}}$ is $\delta^{1/2}$-separated, and $|E_{\delta^{1/2}}| \sim \delta^{-s/2}$. Moreover, $E_{\delta^{1/2}}$ is a generalised $(\delta^{1/2},s)$-set, since for $x \in E_{\delta^{1/2}}$ and $r \geq \delta^{1/2}$, one has
\begin{displaymath} \left(\frac{r}{\delta} \right)^{s} \gtrsim |E \cap B(x,r)| \gtrapprox \delta^{-s/2}|E_{\delta^{1/2}} \cap B(x,r)|. \end{displaymath}

\subsubsection{The distribution of of $P$ in $\delta^{1/2}$-tubes}  Note that \eqref{form5} holds for all $e \in E_{\delta^{1/2}}$. Thus, for every $e \in E_{\delta^{1/2}}$, the set $P_{\delta^{1/2}}$ is covered by a collection of $\leq \delta^{-s/2 - \epsilon_{0}} \lessapprox \delta^{-s/2}$ tubes $\calT_{e}$ of width $\delta^{1/2}$ and perpendicular to $e$. The next goal is to show that, for a typical choice of $e \in E_{\delta^{1/2}}$ and $T \in \calT_{e}$, the set $T \cap P_{\delta^{1/2}}$ is essentially a $(\delta^{1/2},1 - s)$-set. This is a consequence of the next estimate:
\begin{align*} \frac{1}{|E_{\delta^{1/2}}|} \sum_{e \in E_{\delta^{1/2}}} & \sum_{T \in \calT_{e}} \mathop{\sum_{p,q \in T \cap P_{\delta^{1/2}}}}_{p \neq q} \frac{1}{|p - q|^{1 - s}}\\
& \leq \frac{1}{|E_{\delta^{1/2}}|} \mathop{\sum_{p,q \in P_{\delta^{1/2}}}}_{p \neq q} \frac{1}{|p - q|^{1 - s}} \sum_{e \in E_{\delta^{1/2}}} \chi_{\{p,q \in T \text{ for some } T \in \calT_{e}\}}\\
& \lessapprox \frac{1}{|E_{\delta^{1/2}}|} \mathop{\sum_{p,q \in P_{\delta^{1/2}}}}_{p \neq q} \frac{1}{|p - q|} \sim_{\log} \delta^{s/2-1}. \end{align*} 
In passing between the second and third line, the (generalised) $(\delta^{1/2},s)$-set property of $E_{\delta^{1/2}}$ was used, while the last "$\sim_{\log}$" equation follows from the cardinality estimate $|E_{\delta^{1/2}}| \sim \delta^{-s/2}$ and the fact that $P_{\delta^{1/2}}$ is a $(\delta^{1/2},1)$-set. By discarding a constant fraction of points from $E_{\delta^{1/2}}$, one may now assume that
\begin{equation}\label{form4} \sum_{T \in \calT_{e}} \mathop{\sum_{p,q \in T \cap P_{\delta^{1/2}}}}_{p \neq q} \frac{1}{|p - q|^{1 - s}} \lessapprox \delta^{s/2 - 1} \end{equation}
holds uniformly for all $e \in E_{\delta^{1/2}}$. 

\subsubsection{Analysis at scale $\delta$} For $B \in \calB$, recall that $P_{B} = P \cap B$ is a $(\delta,1)$-set of cardinality $|P_{B}| \approx \delta^{-1/2}$. I now claim the following: for fixed $B \in \calB$, there are $\sim \delta^{-s/2}$ vectors in $E_{\delta^{1/2}}$ such that $\pi_{e}(P_{B})$ contains a $(\delta,s)$-set of cardinality $\approx \delta^{-s/2}$.\footnote{The claim is close to Proposition \ref{kaufmanProp}: the main difference is that Proposition \ref{kaufmanProp} only requires the set of directions $E$ to be $\delta$-separated and of cardinality $\approx \delta^{-s}$ (as opposed to being a $(\delta,s)$-set), but also the conclusion there does not guarantee that $\pi_{e}(P)$ would contain a large $(\delta,s)$-set for any $e \in E$. In fact, easy examples show that a cardinality estimate on $E$ alone does not yield the stronger conclusion desired here.} This is fairly standard, but I record the details for completeness. First, observe that $\delta^{-1/2}P_{B}$ is a $(\delta^{1/2},1)$-set of cardinality $\approx \delta^{-1/2}$. Next, consider the measures
\begin{displaymath} \mu_{B} := \frac{1}{|P_{B}|} \sum_{p \in \delta^{-1/2}P_{B}} \frac{\chi_{B(p,\delta^{1/2})}}{\delta} \end{displaymath}
and
\begin{displaymath} \nu := \frac{1}{|E_{\delta^{1/2}}|} \sum_{e \in E_{\delta^{1/2}}} \frac{\chi_{B(e,\delta^{1/2}) \cap S^{1}}}{\delta^{1/2}}, \end{displaymath}
and note that $\mu_{B}(\R^{2}) \sim 1 \sim \nu(S^{1})$. For $r \geq \delta^{1/2}$, one has the uniform estimates $\mu(B(x,r)) \lessapprox r$ and $\nu(B(e,r)) \lesssim r^{s}$, while for $0 < r \leq \delta^{1/2}$ one has the obvious improved estimates. After some straightforward computations, it follows that
\begin{equation}\label{form29} \int_{S^{1}} I_{s}(\pi_{e\sharp}\mu) \, d\nu e := \iint \left[ \int_{S^{1}} \frac{d\nu e}{|\pi_{e}(x) - \pi_{e}(y)|^{s}} \right] \, d\mu x \, d\mu y \lessapprox 1. \end{equation}
Indeed, the inner integral (in brackets) can be estimated by $\lesssim \log(1/\delta)/|x - y|^{s}$, and then
\begin{displaymath} \int_{S^{1}} I_{s}(\pi_{e\sharp}\mu) \, d\nu e \lesssim \log(1/\delta) \int \left[ \frac{d\mu x}{|x - y|^{s}} \right] \, d\mu y \lessapprox 1, \end{displaymath}
since the inner integral is again bounded by $\lessapprox 1$ for any $y \in \R^{2}$. Consequently, $I_{s}(\pi_{e\sharp}\mu) \lessapprox 1$ for a set of vectors $E_{0} \subset S^{1}$ of $\nu$-measure at least $1/2$. One evidently needs $\gtrsim \delta^{-s/2}$ arcs of the form $B(e,\delta^{1/2}) \cap S^{1}$, $e \in E_{\delta^{1/2}}$ to cover $E_{0}$, and this gives rise to a subset $E_{\delta^{1/2}}^{0} \subset E_{\delta^{1/2}}$ with $|E_{\delta^{1/2}}^{0}| \sim \delta^{-s/2}$. For every $e \in E_{\delta^{1/2}}^{0}$, there exists a vector $e' \in B(e,\delta^{1/2}) \cap S^{1}$ with $I_{s}(\pi_{e'\sharp}\mu) \lessapprox 1$. It follows that $\calH^{s}_{\infty}(\pi_{e'}(\spt \mu)) \gtrapprox 1$, hence $\pi_{e'}(\spt \mu)$ contains a $(\delta^{1/2},s)$-set of cardinality $\approx \delta^{-s/2}$ by Proposition \ref{deltasSet}. Since $\pi_{e'}(\spt \mu)$ is contained in the $\delta^{1/2}$-neighbourhood of $\pi_{e'}(\delta^{-1/2}P_{B})$, the same conclusion holds for $\pi_{e'}(\delta^{-1/2}P_{B})$. Finally, using $|e' - e| \leq \delta^{1/2}$, the conclusion remains valid for $\pi_{e}(\delta^{-1/2}P_{B})$, and thus $\pi_{e}(P_{B})$ contains a $(\delta,s)$-set of cardinality $\approx \delta^{-s/2}$ for every $e \in E_{\delta^{1/2}}^{0}$. 

Now, let $G \subset \calB \times E_{\delta^{1/2}}$ consist of those pairs $(B,e)$ such that $\pi_{e}(P_{B})$ contains a $(\delta,s)$-set of cardinality $\approx \delta^{-s/2}$. Then, the previous argument shows that $|\{e \in E_{\delta^{1/2}} : (B,e) \in G\}| \gtrsim \delta^{-s/2}$ for every $B \in \calB$, and consequently
\begin{displaymath} \frac{1}{|E_{\delta^{1/2}}|} \sum_{e \in E_{\delta^{1/2}}} \sum_{B \in  \calB} \chi_{G}(B,e) \gtrsim |\calB| = |P_{\delta^{1/2}}| \gtrsim_{\log} \delta^{-1/2}. \end{displaymath}
This implies that $|\{B \in \calB : (B,e) \in G\}| \gtrsim_{\log} \delta^{-1/2}$ for some $e = e_{0} \in E_{\delta^{1/2}}$. From this point on, the reader may forget about the rest of the vectors in $E_{\delta^{1/2}}$. Let 
\begin{equation}\label{form32} P_{\delta^{1/2}}^{0} := \{p_{B} : (B,e_{0}) \in G\}, \end{equation} 
so that $|P_{\delta^{1/2}}^{0}| \gtrsim_{\log} \delta^{-1/2}$. 
Recall the family of $\delta^{1/2}$-tubes $\calT := \calT_{e_{0}}$. Let $\calT^{\epsilon_{0}} := \{T \in \calT : |T \cap P_{\delta^{1/2}}^{0}| \geq \delta^{2\epsilon_{0} + (s - 1)/2}\}$. Recalling $|\calT| \lesssim \delta^{-s/2 - \epsilon_{0}}$ (by \eqref{form5}), observe that
\begin{displaymath} \sum_{T \in \calT \setminus \calT^{\epsilon_{0}}} |T \cap P_{\delta^{1/2}}^{0}| \lesssim \delta^{-s/2 - \epsilon_{0}} \cdot \delta^{2\epsilon_{0} + (s - 1)/2} = \delta^{\epsilon_{0}-1/2}.  \end{displaymath}
Since $|P_{\delta^{1/2}}^{0}| \gtrsim_{\log} \delta^{-1/2}$, this implies that, for small enough $\delta > 0$, at least $|P_{\delta^{1/2}}^{0}|/2$ points of $P_{\delta^{1/2}}^{0}$ are contained in the union of the tubes in $\calT^{\epsilon_{0}}$. Replacing $\calT$ by $\calT^{\epsilon_{0}}$, I may -- and will -- henceforth assume that $|T \cap P^{0}_{\delta^{1/2}}| \gtrapprox \delta^{(s - 1)/2}$ holds uniformly for all the tubes in $\calT$.

As a further regularisation, I claim that $\gtrapprox |P_{\delta^{1/2}}^{0}|$ points in $P_{\delta^{1/2}}^{0}$ are contained in tubes $T \in \calT$ satisfying the converse inequality $|T \cap P_{\delta^{1/2}}^{0}| \lessapprox \delta^{(s - 1)/2}$. This follows from \eqref{form4}. Recalling that $|T \cap P_{\delta^{1/2}}^{0}| \gtrapprox \delta^{(s - 1)/2}$ for every $T \in \calT$, there exists $C \gtrapprox \delta^{(s - 1)/2}$ such that the "$C$-dense tubes" $T \in \calT^{C}$, with $|P_{\delta^{1/2}}^{0} \cap T| \sim C$, cover a total of $\gtrapprox |P_{\delta_{1/2}}^{0}|$ points in $P_{\delta^{1/2}}^{0}$. Then $|\calT^{C}| \approx |P_{\delta^{1/2}}^{0}|/C \approx \delta^{-1/2}/C$, and hence, by \eqref{form4},
\begin{displaymath} \delta^{s/2 - 1} \gtrapprox \sum_{T \in \calT^{C}} \mathop{\sum_{p,q \in T \cap P_{\delta^{1/2}}^{0}}}_{p \neq q} \frac{1}{|p - q|^{1 - s}} \gtrsim \sum_{T \in \calT_{e}^{C}} |P^{0}_{\delta^{1/2}} \cap T|^{2} \approx \frac{C^{2}}{\delta^{1/2}C} = \frac{C}{\delta^{1/2}}. \end{displaymath} 
This gives $C \approx \delta^{(s - 1)/2}$, as claimed. It has now been established that $\approx |P_{\delta^{1/2}}^{0}| \sim \delta^{-1/2}$ points of $P_{\delta^{1/2}}^{0}$ are covered by tubes $T \in \calT$ satisfying 
\begin{equation}\label{form7} |P_{\delta^{1/2}}^{0} \cap T| \approx \delta^{(s - 1)/2}. \end{equation}
In particular, this implies that there are $\approx \delta^{-s/2}$ tubes in $\calT$ satisfying \eqref{form7}. Finally, using Chebyshev's inequality and \eqref{form4}, one sees that $\approx \delta^{-s/2}$ out of the tubes satisfying \eqref{form7} also satisfy
\begin{equation}\label{form6} \mathop{\sum_{p,q \in T \cap P_{\delta^{1/2}}^{0}}}_{p \neq q} \frac{1}{|p - q|^{1 - s}} \lessapprox \delta^{s - 1}. \end{equation}
In the sequel, I am only interested in the tubes $T$ satisfying both \eqref{form7} and \eqref{form6}. There are $\approx \delta^{-s/2}$ such tubes, and they cover $\approx \delta^{-1/2}$ points of $P_{\delta^{1/2}}^{0}$. For notational convenience, I will continue denoting these tubes by $\calT$.

For $T \in \calT$, write
\begin{displaymath} P_{T} := \bigcup_{p \in T \cap P_{\delta^{1/2}}^{0}} P_{B_{p}}, \end{displaymath}
where $B_{p} \in \calB$ is the unique $\delta^{1/2}$-ball containing $p$ (thus $p = p_{B_{p}}$). Let $E_{\delta}$ be a maximal $\delta$-separated set inside $E \cap B(e_{0},\delta^{1/2})$. Recall from Section \ref{Edelta} that $E_{\delta}$ is a $(\delta,s)$-set with $|E_{\delta}| \approx \delta^{-s/2}$. For future reference, I already observe that
\begin{equation}\label{form30} \pi_{e}(P_{B}) \subset [\pi_{e_{0}}(P_{B})](C\delta), \quad e \in E_{\delta}, \: B \in \calB, \end{equation} 
for some absolute constant $C \geq 1$, where $A(\rho)$ stands for the $\rho$-neighbourhood of $A$. This follows from elementary geometry, recalling that $|e - e_{0}| \leq \delta^{1/2}$ and $\diam(B) = \delta^{1/2}$ for $B \in \calB$. Note that, for $e \in E_{\delta}$, the sets $\pi_{e}(P_{T})$, $T \in \calT$, have bounded overlap. Consequently, recalling also \eqref{form22},
\begin{displaymath} \delta^{-s - 2\epsilon_{0}} \geq N(\pi_{e}(P),\delta) \gtrsim \sum_{T \in \calT} N(\pi_{e}(P_{T}),\delta). \end{displaymath}
Hence
\begin{displaymath} \frac{1}{|E_{\delta}|} \sum_{T \in \calT} \sum_{e \in E_{\delta}} N(\pi_{e}(P_{T}),\delta) \lesssim \delta^{-s - 2\epsilon_{0}}. \end{displaymath}
Since $|\calT| \approx \delta^{-s/2} \sim |E_{\delta}|$, it follows that there is a tube $T_{0} \in \calT$ with 
\begin{equation}\label{form8} \sum_{e \in E_{\delta}} N(\pi_{e}(P_{T_{0}}),\delta) \lessapprox \delta^{-s}. \end{equation}
By Chebyshev's inequality, there exist $\sim \delta^{-s/2}$ vectors $e \in E_{\delta}$ with 
\begin{equation}\label{form9} N(\pi_{e}(P_{T_{0}}),\delta) \lessapprox \delta^{-s/2}. \end{equation}
Such vectors form a $(\delta,s)$-subset of $E_{\delta}$ with cardinality $\sim |E_{\delta}|$, so one may just as well assume that every vector in $E_{\delta}$ satisfies \eqref{form9}. 

Specify one of the vectors in $E_{\delta}$, say $e_{1} \in E_{\delta}$. For notational convenience, assume that
\begin{equation}\label{form10} e_{1} = (1,0). \end{equation}
By the definition \eqref{form32} of $P_{\delta^{1/2}}^{0}$, every projection $\pi_{e_{0}}(P_{B})$ with $p \in T_{0} \cap P_{\delta^{1/2}}^{0}$ contains a $(\delta,s)$-set of cardinality $\approx \delta^{-s/2}$. Moreover, by the simple geometric observation \eqref{form30}, the same remains true for $e_{1}$ in place of $e_{0}$. I denote by $\Delta_{B}$ a $(\delta,s)$-set with $\Delta_{B} \subset \pi_{e_{1}}(P_{B})$ and $|\Delta_{B}| \approx \delta^{-s/2}$.  

Recall the inequality \eqref{form6}, and that $|P_{\delta^{1/2}}^{0} \cap T_{0}| \approx \delta^{(s - 1)/2}$ by \eqref{form7}. Using Chebyshev's inequality, one can now choose a subset $P_{\delta^{1/2}}^{T_{0}} \subset P_{\delta^{1/2}}^{0} \cap T_{0}$ of cardinality $|P_{\delta^{1/2}}^{T_{0}}| \approx \delta^{(s - 1)/2}$ such that
\begin{displaymath} \mathop{\sum_{q \in P_{\delta^{1/2}}^{T_{0}}}}_{q \neq p} \frac{1}{|q - p|^{1 - s}} \lessapprox \delta^{(s - 1)/2}, \qquad p \in T_{\delta^{1/2}}^{T_{0}}. \end{displaymath}
In particular, this implies that $P_{\delta^{1/2}}^{T_{0}}$ is a $(\delta^{1/2},1 - s)$-set with cardinality $\approx \delta^{(s - 1)/2}$.

\subsubsection{Constructing a product-like set with small projections} I first define a set $A_{1}$, which will play the role of "$B$", once Proposition \ref{productProp} is eventually applied. As defined in the previous section, the set $P_{\delta^{1/2}}^{T_{0}}$ is a $(\delta^{1/2},1 - s)$-set with cardinality $\approx \delta^{(s - 1)/2}$. Its points are contained in the $\delta^{1/2}$-tube $T_{0}$ perpendicular to $e_{0}$. Since $|e_{0} - e_{1}| \leq \delta^{1/2}$, and I assumed in \eqref{form10} that $e_{1} = (1,0)$, this means that the projection to the $y$-axis restricted to $P_{\delta^{1/2}}^{T_{0}}$ is "nearly biLipschitz". In particular, the following holds. Write $p_{B} = (p_{B}^{x},p_{B}^{y})$. Then $\{p_{B}^{y} : p_{B} \in P_{\delta^{1/2}}^{T_{0}}\}$ contains a $(\delta^{1/2},1 - s)$-set $A_{1}$ of cardinality $|A_{1}| \approx \delta^{(s - 1)/2}$.


What follows next is a construction of a "product-like" set $F'$ with $|F'| \approx \delta^{-1/2}$ and $N(\pi_{e}(F'),\delta) \lesssim N(\pi_{e}(P_{T_{0}}),\delta)$ for $e \in E_{\delta}$. The set $F'$ will (essentially) play the role of "$P$", once Proposition \ref{productProp} is eventually applied.

The set $F'$ is of the form 
\begin{equation}\label{FPrime} F' = \bigcup_{b \in A_{1}} A_{b}' \times \{b\}, \quad A_{b}' \subset \R. \end{equation}
So, I fix $b \in A_{1}$ and define $A_{b}'$. Note that $b = p_{B}^{y}$ for some $B \in \calB_{0}$. Thus, recalling the $(\delta,s)$-set $\Delta_{B} \subset \pi_{e_{1}}(P_{B})$, I define
\begin{equation}\label{form26} A_{b}' := \Delta_{B}. \end{equation}
Since $|A_{b}'| = |\Delta_{B}| \approx \delta^{-s/2}$, and $|A_{1}| \approx \delta^{(s - 1)/2}$, the estimate $|F'| \approx \delta^{(s - 1)/2}\delta^{-s/2} = \delta^{-1/2}$ holds, as required. 

Next, it is time to control the projections of $F'$. Precisely, the claim is that
\begin{equation}\label{form12} N(\pi_{e}(F'),\delta) \lesssim N(\pi_{e}(P_{T_{0}}),\delta) \lessapprox \delta^{-s/2}, \quad e \in E_{\delta}. \end{equation} 
Recall from \eqref{form9} that $E_{\delta} \subset E \cap B(e_{0},\delta^{1/2}) \subset E \cap B(e_{1},2\delta^{1/2})$ is a set of cardinality $|E_{\delta}| \sim \delta^{-s/2}$ such that the second inequality in \eqref{form12} holds for all $e \in E_{\delta}$. 

To establish the first inequality, it suffices to prove the following: for every $e \in B(e_{1},2\delta^{1/2})$ and every point $q \in F'$, there is a point $p \in P_{T_{0}}$ such that $|\pi_{e}(q) - \pi_{e}(p)| \lesssim \delta$. This follows easily from the construction. Every point of $F'$ is of the form $q = (a,b)$, where $b = p_{B}^{y}$, and $a \in \Delta_{B} \subset \pi_{e_{1}}(P_{B})$. Consequently, there exists a point
\begin{displaymath} p \in P_{B} \subset P_{T_{0}} \end{displaymath}
such that $\pi_{e_{1}}(p) = a = \pi_{e_{1}}(q)$ and $|p - q| \lesssim \delta^{1/2}$. Moreover, it follows from $|p - q| \lesssim \delta^{1/2}$ that
\begin{displaymath} e \mapsto \pi_{e}(p) - \pi_{e}(q) = \pi_{e}(p - q) \end{displaymath}
only varies on an interval of length $\lesssim \delta$, as $e$ varies in $B(e_{1},2\delta^{1/2})$. This and the equation $\pi_{e_{1}}(p) = \pi_{e_{1}}(q)$ imply that $|\pi_{e}(p) - \pi_{e}(q)| \lesssim \delta$ for every $e \in B(e_{1},2\delta^{1/2})$, as required. The estimate \eqref{form12} has been established. 

\subsubsection{Dilating the product set and concluding the proof} The main accomplishment so far has been the construction of the set $F'$ of the form \eqref{FPrime}, which, by \eqref{form12}, has plenty of small projections. This almost looks like a scenario, where Theorem \ref{productProp} can be applied. In fact, all that remains is "normalisation" in terms of a horizontal dilatation.

To this end, it is convenient to re-parametrise the projections $\pi_{e}$, $e \in E_{\delta}$, as mappings of the form $\pi_{t}(x,y) = x + ty$. This is entirely standard, but here are the details: given $e = (\cos \theta,\sin \theta) \in E_{\delta} \in B(e_{1},2\delta^{1/2})$, note that $|\theta| \lesssim \delta^{1/2}$ by the assumption $e_{1} = (1,0)$. Hence one may assume that $\cos \theta \geq 1/2$, and
\begin{equation}\label{form14} \pi_{e}(x,y) = (x,y) \cdot (\cos \theta,\sin \theta) = \tfrac{1}{\cos \theta}\left[x + \tfrac{\sin \theta}{\cos \theta}y \right] =: \tfrac{1}{\cos \theta} \pi_{t(e)}(x,y). \end{equation}
Now, the information that $E_{\delta}$ is a $(\delta,s)$-set contained in an arc of length $\sim \delta^{1/2}$ translates to the statement that $\tilde{E}_{\delta} := \{t(e) : e \in E_{\delta}\}$ is a $(\delta,s)$-set contained in an interval around the origin with length $\sim \delta^{1/2}$. Note that $|\tilde{E}_{\delta}| \sim |E_{\delta}| \sim \delta^{-s/2}$, and I assume for convenience assume that $\tilde{E}_{\delta} \subset [0,\delta^{1/2}]$. Moreover, it is obvious from \eqref{form12} and the formula \eqref{form14} that
\begin{equation}\label{form15} N(\pi_{t}(F'),\delta) \lessapprox \delta^{-s/2}, \qquad t \in \tilde{E}_{\delta}. \end{equation}
Finally, a horizontal dilatation is applied. Consider the set 
\begin{displaymath} F := \{(\delta^{-1/2}x,y) : (x,y) \in F'\} = \bigcup_{b \in A_{1}} (\delta^{-1/2}A_{b}') \times \{b\}. \end{displaymath}
To complete the proof, observe that each set $A_{b} := \delta^{-1/2}A_{b}'$, $b \in A_{1}$, is a $(\delta^{1/2},s)$-set of cardinality $\approx \delta^{-s/2}$. The cardinality claim is clear from the definition \eqref{form26}, while the $(\delta^{1/2},s)$-set property follows from
\begin{equation}\label{form27} |B(x,r) \cap A_{b}| = |B(\delta^{1/2}x,\delta^{1/2}r) \cap A_{b}'| \lessapprox \left(\frac{\delta^{1/2}r}{\delta}\right)^{s} = \left(\frac{r}{\delta^{1/2}}\right)^{s}, \quad r \geq \delta^{1/2}. \end{equation}
Consequently, the sets $A_{1}$ (as "$B$"), $A_{b}$, $b \in A_{1}$, and $F$ (as "$P$") satisfy the hypotheses of Proposition \ref{productProp} at scale $\delta^{1/2}$. Moreover, $F$ has plenty of small projections. Consider the set $\tilde{E}_{\delta^{1/2}} := \delta^{-1/2}\tilde{E}_{\delta} \subset [0,1]$. Given $t' = \delta^{-1/2}t \in \tilde{E}_{\delta^{1/2}}$ and $(\delta^{-1/2}x,y) \in F$, observe that
\begin{displaymath} \pi_{t'}(\delta^{-1/2}x,y) = \delta^{-1/2}x + \delta^{-1/2}ty = \delta^{-1/2}\pi_{t}(x,y). \end{displaymath}
Recalling \eqref{form15}, it follows immediately that
\begin{equation}\label{form28} N(\pi_{t'}(F),\delta^{1/2}) = N(\pi_{t}(F'),\delta) \lessapprox \delta^{-s/2}, \quad t' \in \tilde{E}_{\delta^{1/2}}. \end{equation}
The set $\tilde{E}_{\delta^{1/2}}$ is clearly (or see \eqref{form27}) a $(\delta^{1/2},s)$-set with $|\tilde{E}_{\delta^{1/2}}| \sim \delta^{-s/2}$. Consequently, \eqref{form28} should not be possible by Proposition \ref{productProp}. A contradiction is thus reached, and the proof of Theorem \ref{main} is complete.

\end{document}